\documentclass{article}
\usepackage{hyperref}
\usepackage{amssymb}
\usepackage{amsmath}
\usepackage{amsthm}
\usepackage{enumerate}
\usepackage{graphicx}
\usepackage[usenames,dvipsnames]{xcolor}

\newtheorem{case}{Case}

\theoremstyle{plain}

\makeatletter
\newtheorem*{rep@theorem}{\rep@title}
\newcommand{\newreptheorem}[2]{%
\newenvironment{rep#1}[1]{%
 \def\rep@title{#2 \ref{##1}}%
 \begin{rep@theorem}}%
 {\end{rep@theorem}}}
\makeatother

\newtheorem{theorem}{Theorem}
\newreptheorem{theorem}{Theorem}

\newtheorem{lemma}[theorem]{Lemma}
\newtheorem{corollary}[theorem]{Corollary}

\newtheorem*{definition}{Definition}

\title{Characterizing meromorphic psuedo-lemniscates}
\author{Trevor Richards}
\date{July 2016}

\begin{document}

\maketitle

\begin{abstract}
Let $f$ be a meromorphic function with simply connected domain $G\subset\mathbb{C}$, and let $\Gamma\subset\mathbb{C}$ be a smooth Jordan curve.  We call a component of $f^{-1}(\Gamma)$ in $G$ a $\Gamma$-\textit{pseudo-lemniscate} of $f$.  In this note we give criteria for a smooth Jordan curve $\mathcal{S}$ in $G$ (with bounded face $D$) to be a psuedo-lemniscate of $f$ in terms of the number of preimages (counted with multiplicity) which a given $w$ has under $f$ in $D$, as $w$ ranges over the Riemann sphere.  We also develop a test, in the same terms, by which one may show that the image of a Jordan curve under $f$ is not a Jordan curve.
\end{abstract}

\section{Introduction}

It is well known that a degree $n$ Blaschke product is an $n$-to-$1$ (counting multiplicity) analytic self-map of the unit disk $\mathbb{D}$, and moreover that every such analytic $n$-to-$1$ self-map of the unit disk is a degree $n$ Blaschke product.  We can restate this in a slightly more general setting.  Throughout this paper we let $G\subset\mathbb{C}$ denote a simply connected domain, and we let $\mathcal{S}\subset G$ be a smooth Jordan curve, with bounded face $D$.

\begin{theorem}\label{thm: Blaschke prod. version.}
Let $f:G\to\mathbb{C}$ be analytic.  The following are equivalent.

\begin{enumerate}
\item $\mathcal{S}$ is a lemniscate of $f$ in $G$.
\item For some positive integer $n$, $f$ maps $D$ precisely $n$-to-$1$ onto $\mathbb{D}$.
\end{enumerate}

If Items $1$ and $2$ hold, the function $f\circ\phi$ is a degree $n$ Blaschke product, where $\phi:\mathbb{D}\to D$ is any Riemann map.
\end{theorem}

The purpose of this note is to generalize the result of Theorem~\ref{thm: Blaschke prod. version.} to the context of meromorphic functions and pseudo-lemniscates.  In order to state the theorem, we begin with several definitions.  For the remainder of the paper we let $\Gamma$ denote some smooth Jordan curve in $\mathbb{C}$, with bounded face $\Omega_-$ and unbounded face $\Omega_+$ (where $\Omega_+$ includes the point $\infty$).

\begin{definition}\ 
\begin{itemize}
\item Let $f:G\to\hat{\mathbb{C}}$ be meromorphic.  Suppose that $\mathcal{S}$ is a component of $f^{-1}(\Gamma)$.  We call $\mathcal{S}$ a $\Gamma$-\textit{pseudo-lemniscate} of $f$.
\item For any $w\in\hat{\mathbb{C}}$, let $\mathcal{N}_f(w)$ denote the number of preimages under $f$ of $w$ in $D$, counted with multiplicity.
\end{itemize}
\end{definition}

In the special case that $\Gamma=\mathbb{T}$, the above definition reduces to the classical definition of a lemniscate.  The definition of a pseudo-lemniscate was introduced in~\cite{RY} wherein conformal equivalence of functions on domains bounded by pseudo-lemniscates was studied.  We now proceed to the promised generalization of Theorem~\ref{thm: Blaschke prod. version.}.

\begin{theorem}\label{thm: Main.}
Let $f:G\to\hat{\mathbb{C}}$ be meromorphic.  The following are equivalent.

\begin{enumerate}
\item $\mathcal{S}$ is a $\Gamma$-pseudo-lemniscate of $f$ in $G$.
\item For some non-negative integers $n_-$ and $n_+$, the following holds.
\begin{itemize}
\item For every $w\in\Omega_-$, $\mathcal{N}_f(w)=n_-$.
\item For every $w\in\Omega_+$, $\mathcal{N}_f(w)=n_+$.
\item For every $w\in\Gamma$, $\mathcal{N}_f(w)=\min(n_-,n_+)$.
\end{itemize}
\end{enumerate}

If Items $1$ and $2$ hold, and $f$ is analytic in $D$, then $n_+=0$, and the function $\psi^{-1}\circ f\circ\phi$ is a degree $n_-$ Blaschke product, where $\phi:\mathbb{D}\to D$ and $\psi:\mathbb{D}\to\Omega_-$ are any Riemann maps.

If Items $1$ and $2$ hold, and $\Gamma=\mathbb{T}$, then $f\circ\phi$ is a ratio of Blaschke products (degree $n_-$ in the numerator and degree $n_+$ in the denominator), where $\phi:\mathbb{D}\to D$ is any Riemann map.

\end{theorem}

Using the fact that any meromorphic function is conformal away from its critical points, we may use Theorem~\ref{thm: Main.} to supply a test for when the image of $\mathcal{S}$ under $f$ will not be a Jordan curve.

\begin{corollary}\label{cor: Non-Jordan curve or crit. point.}
Let $f:G\to\hat{\mathbb{C}}$ be meromorphic.  If there are three points $w_1,w_2,w_3\in\hat{\mathbb{C}}$ such that the numbers $\mathcal{N}_f(w_1)$, $\mathcal{N}_f(w_2)$, and $\mathcal{N}_f(w_3)$ are distinct, then either $f$ has a critical point on $\mathcal{S}$ or $f(\mathcal{S})$ is not a Jordan curve.
\end{corollary}

In Section~\ref{sect: Lemmas and proof.} we will provide proofs of Theorem~\ref{thm: Main.} and Corollary~\ref{cor: Non-Jordan curve or crit. point.}.  In Section~\ref{sect: Example.} we apply Corollary~\ref{cor: Non-Jordan curve or crit. point.} to show that the image of a certain smooth Jordan curve $\mathcal{S}$ under a certain meromorphic function $f$ is not a Jordan curve, without explicitly finding a point of intersection in $f(\mathcal{S})$.

\section{Lemmas and Proofs}\label{sect: Lemmas and proof.}

We begin with a topological observation on the convergence of preimages under continuous maps.  First we introduce a bit of standard notation.

\begin{definition}
For any $x\in\mathbb{C}$, $X\subset\mathbb{C}$, and $\delta>0$, we define $$B(x;\delta)=\{y\in\mathbb{C}:|x-y|<\delta\}$$ and $$B(X;\delta)=\displaystyle\bigcup_{z\in X}B(z;\delta).$$
\end{definition}

\begin{lemma}\label{lem: Convergence of preimages.}
Let $U\subset\mathbb{C}$ be an open bounded set, and let $h:U\to\hat{\mathbb{C}}$ be continuous.  Choose some $w_0\in\mathbb{C}$.  Then for any $\delta>0$ there exists an $\epsilon>0$ such that if $w\in B(w_0;\epsilon)$, then $h^{-1}(w)\subset B(h^{-1}(w_0)\cup\partial U;\delta)$.
\end{lemma}

\begin{proof}
This is an elementary exercise in compactness and continuity.
\end{proof}

Our second lemma shows that, under the conditions of Item~2 from Theorem~\ref{thm: Main.}, if $E$ is a component of $f^{-1}(\Omega_-)\cap D$, and $z\in\partial E$, then $f(z)\in\partial\Omega_-$.

\begin{lemma}\label{lem: Full copies.}
Let $h:G\to\hat{\mathbb{C}}$ be meromorphic.  Let $\Omega\subset\mathbb{C}$ be open, bounded, and connected, and assume that the quantity $\mathcal{N}_h(w)$ is independent of the choice of $w\in\Omega$.  Then if $E\subset D$ is a component of the set $h^{-1}(\Omega)\cap D$, and $z_0\in\partial E$, then $h(z_0)\in\partial\Omega$.
\end{lemma}

\begin{proof}

Let $h$, $\Omega$, and $E$ be as in the statement of the lemma.  Choose some $z_0\in\partial E$, and set $w_0=h(z_0)$.  Suppose by way of contradiction that $w_0\in\Omega$.  We proceed by cases.

\begin{case}$z_0\in D$.\end{case}

Since $h$ is continuous and $\Omega$ and $D$ are open, there is a neighborhood of $z_0$ contained in $h^{-1}(\Omega)\cap D$, contradicting our choice of $z_0$.

\begin{case}$z_0\notin D$.\end{case}

Set $m=\mathcal{N}_h(w_0)$, and let $z_1,z_2,\ldots,z_k$ be the distinct preimages of $w_0$ in $D$, having multiplicities $m_1,m_2,\ldots,m_k$ (so that $m=m_1+m_2+\cdots+m_k$).  Choose some $\delta>0$ such that $\delta<\min(|z_i-z_0|:1\leq i\leq k)$.  Choose disjoint neighborhoods $D_1,D_2,\ldots,D_k\subset D$ such that for each $1\leq i\leq k$, the following holds.

\begin{itemize}
\item $z_i\in D_i$.
\item $D_i\cap B(z_0;\delta)=\emptyset$.
\item $h$ is exactly $m_i$-to-$1$ on $D_i\setminus\{z_i\}$.
\end{itemize}

Choose some point $x\in B(z_0;\delta)\cap E$ sufficiently close to $z_0$ so that $$h(x)\in\displaystyle\bigcap_{i=1}^kh(D_i).$$  Then $h^{-1}(h(x))\cap D$ contains $x$, along with $m_i$ distinct points in $D_i$, for each $1\leq i\leq k$.  Thus since $x\notin D_i$ for any $1\leq i\leq k$, we have that $\mathcal{N}_h(h(x))$ is at least $$m_1+m_2+\ldots+m_k+1>m.$$  Since $x\in E$, $h(x)\in\Omega$, so we have a contradiction of the choice of $m$.

We therefore conclude that $w_0\notin\Omega$.  Since $h$ is continuous we must have that $w_0\in\overline{\Omega}$, so we conclude that $w_0\in\partial\Omega$.
\end{proof}

With these lemmas under our belt, we proceed to the proof of Theorem~\ref{thm: Main.}.

\begin{proof}[Proof of Theorem~\ref{thm: Main.}.]
We begin by assuming that $\mathcal{S}$ is a $\Gamma$-pseudo-lemniscate of $f$ in $G$.

Fix now some component $E$ of $f^{-1}(\Omega_-)$ in $D$. We wish to show that there is some positive integer $n_E$ such that $f$ maps $E$ exactly $n_E$-to-$1$ onto $\Omega_-$ (counting multiplicity).

Let $\Lambda_1,\Lambda_2,\ldots,\Lambda_k$ denote the components of $\partial E$, with $\Lambda_1$ being the component of $\partial E$ which is in the unbounded component of $E^c$.  Since $f$ is continuous and open, $f$ maps each $\Lambda_i$ to the boundary of $\Omega_-$, namely to $\Gamma$.  Therefore for each $i$, we define $n_i$ to be the number of times $f(z)$ winds around $\Gamma$ as $z$ traverses $\Lambda_i$ one time with positive orientation.  

By the Riemann mapping theorem, for any point $w\in\Omega_-$, we can find some conformal map $\rho:\Omega_-\to\mathbb{D}$ such that $\rho(w)=0$, and since the boundary of $\Omega_-$ is smooth, each such $\rho$ extends conformally to a neighborhood of $\overline{\Omega_-}$.  Therefore, replacing $f$ with $\rho\circ f$, the components of $\partial E$ are honest lemniscates of $\rho\circ f$.  Since $\rho$ is a conformal map on the closure of $\Omega_-$, if $f(z)$ winds once around $\Gamma$, $\rho\circ f(z)$ winds once around the unit circle, and with the same orientation (since $\Omega_-$ is simply connected and $\rho$ is analytic, $\rho$ cannot reverse the orientation).

Therefore it follows from the argument principle that the number of zeros of $\rho\circ f$ in $E$ (counted with multiplicity) is

$$n_E:=n_1-(n_2+n_3+\cdots+n_k).$$

That is, the number of members of $f^{-1}(w)$ in $E$ (counted with multiplicity) is $n_E$.  Since this number $n_E$ is independent of the choice of $w\in\Omega_-$, we conclude that $f$ maps $E$ precisely $n_E$-to-$1$ onto $\Omega_-$.  Therefore for any $w\in\Omega_-$, $$\mathcal{N}_f(w)=\displaystyle\sum_E n_E.$$

This sum is taken over all components $E$ of $f^{-1}(\Omega_-)$ in $D$, a quantity which again is independent of the choice of $w\in\Omega_-$.  We call this sum $n_-$.  Doing the same calculation with $\Omega_-$ replaced by $\Omega_+$, we obtain the corresponding number $n_+$.

We now consider the set $f^{-1}(\Gamma)\cap D$.  Since $\mathcal{S}$ is a connected component of $f^{-1}(\Gamma)$ and $\mathcal{S}$ is compact, and $f^{-1}(\Gamma)$ is closed in $G$, it follows that $f^{-1}(\Gamma)\setminus\mathcal{S}$ is isolated from $\mathcal{S}$ by some positive distance.  Let $\delta>0$ be chosen to be less than half that distance.  That is, $$\displaystyle\inf_{x\in\mathcal{S},y\in f^{-1}(\Gamma)\setminus\mathcal{S}}|x-y|>2\delta.$$

Since $D$ is simply connected, the set $D_\delta:=\left\{z\in D:d(z,\mathcal{S})<\delta\right\}$ is connected (where $d(z,\mathcal{S})$ denotes the distance $\displaystyle\inf_{x\in\mathcal{S}}|z-x|$).  Therefore since $f$ is continuous, and $f^{-1}(\Gamma)$ does not intersect $D_\delta$, we conclude that either $f(D_\delta)\subset\Omega_-$ or $f(D_\delta)\subset\Omega_+$.  Without loss of generality assume that $f(D_\delta)\subset\Omega_-$.

Choose now some $w_0$ in $\Gamma$.  Let $z_1,z_2,\ldots,z_l\in D$ denote the distinct preimages of $w_0$ in $D$, having multiplicities $m_1,m_2,\ldots,m_l$.  Choose disjoint domains $D_1,D_2,\ldots,D_l\subset D$ such that the following holds for each $i$.

\begin{itemize}
\item $z_i\in D_i$.
\item $D_i\cap D_\delta=\emptyset$.
\item $f$ has no critical points in $D_i\setminus\{z_i\}$.
\item $f$ is exactly $m_i$-to-$1$ in $D_i\setminus\{z_i\}$.
\end{itemize}

Reduce now $\delta>0$ if necessary so that for each $i$, $B(z_i;\delta)\subset D_i$. By Lemma~\ref{lem: Convergence of preimages.} we can find an $\epsilon>0$ such that if $w\in\Omega_+$ is within $\epsilon$ of $w_0$, then $$f^{-1}(w)\cap D\subset B((f^{-1}(w_0)\cap D)\cup\partial D;\delta)\cap D.$$  Choose $\epsilon>0$ smaller if necessary so that for each $i$, $B(w_0;\epsilon)\subset f(D_i)$.

Fix now some $w\in B(w_0;\epsilon)\cap\Omega_+$.  Since $w\in\Omega_+$, and $f(D_\delta)\subset\Omega_-$, no point in $f^{-1}(w)\cap D$ is within $\delta$ of $\partial D$.  Therefore by choice of $\epsilon$ we have $$f^{-1}(w)\cap D\subset B(f^{-1}(w_0)\cap D;\delta)\subset\displaystyle\bigcup_{i=1}^lD_i.$$

Now, for each $i$, since $B(w_0;\epsilon)\subset f(D_i)$, $w\in f(D_i)$.  Since $f$ is exactly $m_i$-to-$1$ in $D_i\setminus\{z_i\}$, there are exactly $m_i$ distinct points in $f^{-1}(w)\cap D_i$, and since none of these points are critical points, each one contributes exactly $1$ to the total $n_+$ many elements in $f^{-1}(w)\cap D$.  We therefore conclude that $\mathcal{N}_f(w_0)=n_+$.

Note that if we had selected the point $w$ to be a point in $\Omega_-$ instead of in $\Omega_+$, we would have obtained the same result, except that some members of $f^{-1}(w)\cap D$ could additionally reside in $D_\delta$.  Our conclusion would then be that the $\mathcal{N}_f(w_0)\leq n_-$.  Since $\mathcal{N}_f(w_0)=n_+$ and $\mathcal{N}_f(w_0)\leq n_-$, we therefore have that $$\mathcal{N}_f(w_0)=\min(n_-,n_+).$$  

If we had adopted the other assumption, that $f(D_\delta)\subset\Omega_+$, then similar work would give the same result.  This concludes the first direction of our proof.

Suppose now that the second item holds.  That is, that there are some non-negative integers $n_-$ and $n_+$ for which the following holds.

\begin{itemize}
\item For every $w\in\Omega_-$, $\mathcal{N}_f(w)=n_-$.
\item For every $w\in\Omega_+$, $\mathcal{N}_f(w)=n_+$.
\item For every $w\in\Gamma$, $\mathcal{N}_f(w)=\min(n_-,n_+)$.
\end{itemize}

It follows from Lemma~\ref{lem: Full copies.} that if $E$ is any component of $f^{-1}(\Omega_-)\cap D$ or of $f^{-1}(\Omega_+)\cap D$, then $f(\partial E)\subset\Gamma$.  We therefore conclude that $f(\mathcal{S})=f(\partial D)\subset \Gamma$.  Our goal is to show that $\mathcal{S}$ itself is an entire component of $f^{-1}(\Gamma)$ in $G$.

For any point $z\in\mathcal{S}$, since $f$ is analytic at $z$ and $\Gamma$ is smooth, if $z$ is not a critical point of $f$ then $f^{-1}(\Gamma)$ is a smooth curve close to $z$.  On the other hand, if $z$ is a critical point of $f$ with multiplicity $m$ then, close to $z$, $f^{-1}(\Gamma)$ consists of $2(m+1)$ paths meeting at $z$, forming equal angles $\dfrac{2\pi}{2(m+1)}$ around $z$.  Since $\mathcal{S}$ is smooth, it therefore suffices to show that our assumption on $f$ implies that $f$ does not have any critical points on $\mathcal{S}$.

Suppose by way of contradiction that $z_0\in\mathcal{S}$ is a critical point of $f$ with multiplicity $m$. Let $\Lambda$ denote the component of $f^{-1}(\Gamma)$ which contains $\mathcal{S}$.  Since $m\geq1$, $$\dfrac{2\pi}{2(m+1)}=\dfrac{\pi}{m+1}<\pi.$$  Therefore since $\mathcal{S}$ is smooth, there are edges of $\Lambda$ which extend from $z_0$ into either face of $\mathcal{S}$, and in particular into $D$.  Let $\gamma$ denote some such edge of $\Lambda$ emanating from $z_0$ into $D$.  Assume without loss of generality that as $z$ traverses $\gamma$, $f(z)$ traverses $\Gamma$ with positive orientation.  Then for any point $w\in\Gamma$ taken arbitrarily close to $f(z_0)$ (with positive orientation around $\Gamma$ from $f(z_0)$), there is some point $z\in D$ (on $\gamma$) which is arbitrarily close to $z_0$, such that $f(z)=w$.  A contradiction now follows from exactly the same reasoning as was used to prove Lemma~\ref{lem: Full copies.}.

We therefore conclude that $\mathcal{S}$ contains no critical point of $f$, and thus that the component of $f^{-1}(\Gamma)$ which contains $\mathcal{S}$ is just $\mathcal{S}$ itself.  That is, $\mathcal{S}$ is a $\Gamma$-pseudo-lemniscate of $f$ in $G$.  This concludes the second direction of the proof.

Now suppose that Items~1~and~2 in the theorem do obtain, and that moreover $f$ is analytic in $D$.  Since $f^{-1}(\infty)\cap D=\emptyset$, immediately we have that $n_+=0$.  If $\phi:\mathbb{D}\to D$ and $\psi:\mathbb{D}\to\Omega_-$ are Riemann maps, then $\psi^{-1}\circ f\circ\phi$ is a $n_-$-to-$1$ analytic self map of $\mathbb{D}$, which by Theorem~\ref{thm: Blaschke prod. version.} is a degree~$n_-$ Blaschke product.

Finally, suppose that Items~1~and~2 in the theorem do obtain, and that $\Gamma=\mathbb{T}$.  Let $\phi:\mathbb{D}\to D$ be any Riemann map.  Since $\partial D$ is smooth, $\phi$ extends conformally to a neighborhood of the closed unit disk.  Since $f$ satisfies Item~2, $f\circ\phi$ has $n_-$ zeros and $n_+$ poles in $\mathbb{D}$, and since $f$ satisfies Item~1, $|f\circ\phi|=1$ on $\mathbb{T}$.  Let $A$ denote a degree $n_-$ Blaschke product having the same zeros as $f\circ\phi$, with the same multiplicities, and let $B$ denote a degree $n_+$ Blaschke product, having zeros exactly at the locations of the poles of $f\circ\phi$, with the same multiplicities.  Then $(f\circ\phi)\cdot\frac{B}{A}$ has no zeros or poles in $\mathbb{D}$ and on $\partial D$, $$\left|(f\circ\phi)\cdot\frac{B}{A}\right|=|f\circ\phi|\cdot\frac{|B|}{|A|}=1.$$  Therefore by the maximum/minimum modulus theorems, $(f\circ\phi)\cdot\frac{B}{A}$ is a unimodular constant $\lambda$.  Replacing $A$ by $\lambda\cdot A$, we have $f\circ\phi=\frac{A}{B}$.

\end{proof}

\begin{proof}[Proof of Corollary~\ref{cor: Non-Jordan curve or crit. point.}.]
Suppose that $f$, $w_1$, $w_2$, and $w_3$ are as in the statement of the lemma.  Suppose that $f(\mathcal{S})$ is a Jordan curve.  Setting $\Gamma=f(\mathcal{S})$, since $\mathcal{N}_f(w_1)$, $\mathcal{N}_f(w_2)$, and $\mathcal{N}_f(w_3)$ are distinct, $\mathcal{S}$ cannot be a $\Gamma$-pseudo-lemniscate of $f$ by Theorem~\ref{thm: Main.}.  Let $\Lambda$ denote the component of $f^{-1}(\Gamma)$ in $G$ which contains $\mathcal{S}$.  Since $\Lambda\neq\mathcal{S}$, $\Lambda$ must branch off from $\mathcal{S}$ at some point, and thus $f$ must not be conformal at that point.  That is, $f$ must have a critical point on $\mathcal{S}$.

\end{proof}

\section{Application of Corollary~\ref{cor: Non-Jordan curve or crit. point.}}\label{sect: Example.}

\textbf{Example.} Let $\mathcal{S}$ denote some smooth curve in $\mathbb{C}$ which contains the quadrilateral with vertices $0$, $1$, $2i$, and $1+4i$, and which approximates the boundary of this shape closely.  Let $D$ denote the bounded face of $\mathcal{S}$.  Then setting $f(z)=e^z$, we claim that $f(\mathcal{S})$ is not a Jordan curve.

To see this, observe that for $w_1=e^0i$, $w_2=e^1i$, and $w_3=e^2i$, we have $\mathcal{N}_f(w_1)=1$, $\mathcal{N}_f(w_2)=2$, and $\mathcal{N}_f(w_3)=0$.  Therefore Corollary~\ref{cor: Non-Jordan curve or crit. point.} tells us that either $f$ has a critical point on $\mathcal{S}$ or $f(\mathcal{S})$ is not a Jordan curve.  Since $f(z)=e^z$ has no critical points, we conclude that $f(\mathcal{S})$ is not a Jordan curve.

\bibliographystyle{amsplain}

\end{document}